\documentclass{article}
\usepackage[pdftex]{hyperref}
\hypersetup{
	colorlinks   = true, 
	urlcolor     = gray, %Color for external hyperlinks
	linkcolor    = red, %Color of internal links
	citecolor   =  blue%Color of citations
}
\usepackage{tikzsymbols,enumitem,framed,overpic}
\usepackage{amsmath,amsfonts,amsthm,amssymb,graphicx,tikz-cd,tikz,chngcntr}
\usepackage{authblk}
\usepackage[all,2cell,ps]{xy}
\usepackage[margin=1.5 in]{geometry}

\title{On Higher Dimensional Milnor Frames}
\counterwithin{figure}{section}
\newcommand{\bb}{\mathbb}
\newcommand{\ad}{\text{ad}}
\newcommand{\where}{|}
\newcommand{\ric}{\text{Ric}}

\newcommand{\Span}{\text{Span}}
\newcounter{Theorem}
\counterwithin{Theorem}{section}
\newtheorem{theorem}[Theorem]{Theorem}
\newtheorem{definition}[Theorem]{Definition}
\newtheorem{proposition}[Theorem]{Proposition}
\newtheorem{lemma}[Theorem]{Lemma}
\newtheorem{question}[Theorem]{Question}
\newtheorem{remark}[Theorem]{Remark}
\newtheorem{example}[Theorem]{Example}
\newtheorem{corollary}[Theorem]{Corollary}

\newcommand{\spa}{\text{Span}}

\author{\small{Hayden Hunter }\\ \scriptsize{University of Florida}\\ Gainesville, FL, USA\\ \footnotesize{\textsf{haydenhunter@ufl.edu}}}
\date{}
\begin{document}
	\maketitle
	\begin{abstract}
		A classic result of Milnor shows that any 3-dimensional unimodular metric Lie algebra admits an orthonormal frame with at most three nontrivial structure constants. These frames are referred to as Milnor frames. We define extensions of Milnor frames into higher dimensions and refer to these higher dimensional analogues as Lie algebras with Milnor frames. We determine that $n$-dimensional, $n \geq 4$, Lie algebras with Milnor frames are isomorphic to the direct sum of 3-dimensional Heisenberg Lie algebras $\mathfrak{h}^3$, 4-dimensional 3-step nilpotent Lie algebras $\mathfrak{h}^4$, and an abelian Lie algebra $\mathfrak{a}$. Moreover, for any Lie algebra $\mathfrak{g}\not\cong \mathfrak{h}^3 \oplus \mathfrak{a}$ with a Milnor frame, there exists an inner product structure $g$ on $\mathfrak{g}$ such that $(\mathfrak{g}, g)$ does not admit an orthonormal Milnor frame. 
	\end{abstract}
	%	\noindent \keywords{Homogeneous, Nilmanifolds, Unimodular, Ricci Flow}
	\newpage
	\tableofcontents
	\newpage
	\setcounter{section}{0}
	\section{Introduction}

	\indent \indent 
	Given a Lie Group $G$ and a (Riemannian) metric\footnote{For the remainder of the paper any use of the word ``metric'' means ``Riemannian metric''} $g$ on $G$, we say that $g$ is left-invariant if, for any $p \in G$, the diffeomorphism 
	\begin{equation}
		L_p: G \to G, \quad q \mapsto pq
	\end{equation}
	is an isometry with respect to $g$. A vector field $X$ on $G$ is said to be left-invariant if $L_p^*(X) = X$ for any $p \in G$. By placing an inner product structure on the tangent plane at the identity, $T_eG$, we can extend this inner product to a left-invariant metric on $G$ (See for example Chapter 1 Section 2 in \cite{MR1138207}). The collection of left-invariant metrics on $G$ are in one-to-one correspondence with inner product structures on $T_eG$. Thus given a left-invariant metric on a simply connected Lie Group $G$, we only need to look at the inner product structure of its associated Lie algebra. With this in mind, whenever a simply connected Lie Group $G$ is equipped with a left-invariant metric $g$, we may call the Lie algebra $(\mathfrak{g},g)$ associated to the Lie Group a metric Lie algebra with respect to the metric $g$. When it is obvious which left-invariant metric we are referring to, we say that $\mathfrak{g}$ is a metric Lie algebra. \\
	
	\indent When $G$ is equipped with the Levi-Civita connection, the curvature is determined by the structure constants induced by a Lie bracket $[\, , \,] \in \wedge^2 \mathfrak{g}^* \otimes \mathfrak{g}$ (See sections 5 and 6 in \cite{MR425012}). Thus if a metric Lie algebra has few nontrivial structure constants, determining curvature requires fewer computations. For example, Milnor \cite[Lemma 4.1]{MR425012} determined that any $3$-dimensional unimodular metric Lie algebra admits an orthonormal frame with at most $3$  structure constants. Moreover, these orthonormal frames diagonalize the Ricci-tensor. For higher dimensional metric Lie algebras, Hashinaga, Tamaru, and Terada \cite{MR3488140} introduced a procedure that allows us to find an orthonormal frame with relatively few structure constants.\\
	
	%		   Given a nilpotent metric Lie algebra and an orthonormal frame, one can determine algebraically which Lie brackets allow the orthonormal frame to diagonalize the Ricci-tensor. This is shown by Lauret and Will in Theorem 1.1~\cite[3652]{MR3080187}.
	\indent Let $\mathfrak{g}$ be a three-dimensional metric Lie algebra with metric $g$ and a fixed orientation. For two linearly independent vectors $X,Y \in \mathfrak{g}$, denote the exterior product $X \wedge Y$ to be the unique vector orthogonal to $X$ and $Y$ such that $\{X,Y, X \wedge Y\}$ is a positively oriented frame and 
	$$
	||X \wedge Y||^2 := g(X\wedge Y, X\wedge Y) = g(X,X)g(Y,Y) - g(X,Y)^2.
	$$
	If $X$ and $Y$ are not linearly independent then we say that $X \wedge Y = 0$. By the universal property of wedge products there exists a linear operator $L: \mathfrak{g} \to \mathfrak{g}$ for which the diagram
	
	\begin{center}
		\begin{tikzpicture}[scale = 2]
			\node(A)at(0,1){$\mathfrak{g}\times \mathfrak{g}$}; 
			\node(C)at(1,1){$\mathfrak{g}$}; 
			\node(D)at(1,0){$\mathfrak{g}$}; 
			\path[->,font=\scriptsize] 
			(A)edgenode[above]{$\wedge$}(C) 
			(A)edgenode[left]{$[\,,\,]$}(D) 
			(C)edgenode[right]{$L$}(D);
		\end{tikzpicture}
	\end{center}
	\noindent commutes. Moreover $\mathfrak{g}$ is unimodular if and only if $L$ is self-adjoint~\cite[Lemma 4.1]{MR1138207}. It is well-known that a self-adjoint linear operator admits an orthonormal frame of eigenvectors. Thus there is an orthonormal frame $\{X_1, X_2,X_3\}$ and real numbers $\lambda_1, \lambda_2,\lambda_3 \in \bb R$ such that 
	$$
	[X_1,X_2] = \lambda_3X_3, \quad [X_2,X_3] = \lambda_1X_1, \quad [X_3,X_1] = \lambda_2X_2. 
	$$
	Let $\sigma = (1 \, 2 \, 3) \in S_3$ where $S_3$ is the permutation group acting on $3$ elements. We can rewrite the bracket identity so that for any $i \in \{1,2,3\}$, $[X_i, X_{\sigma(i)}] = \lambda_{\sigma^2(i)} X_{\sigma^2(i)}$. 
	
	\begin{definition}
		Let $\mathfrak{g}$ be a Lie algebra with a frame $\{X_1, X_2,X_3\}$. If there exists $\lambda_1,\lambda_2,\lambda_3 \in \bb R$ such that the bracket relation $[X_i,X_{\sigma(i)}] = \lambda_{\sigma^2(i)} X_{\sigma^2(i)}$ holds for any $1 \leq i \leq 3$, then $\{X_1,X_2,X_3\}$ is called a Milnor frame. 
	\end{definition}
	
	\indent Using the permutation $\sigma = (1 \,\ldots \, n)$ we can extend this notion to higher dimensions. 
	\begin{definition}\label{HDMFDef}
		Let $\mathfrak{g}$ be a Lie algebra of dimension $n \geq 4$ and $\sigma = (1 \, \ldots \, n) \in S_n$ where $S_n$ is the permutation group acting on $n$ elements. A frame, $\{X_1, \ldots, X_n\}$, is a Milnor frame if for all $i,j \in \{1, \ldots, n\}$, there exist $\lambda_1, \ldots, \lambda_n \in \bb R$ such that
		\begin{equation}
			[X_i,X_j] = \begin{cases} \lambda_{\sigma^2(i)}X_{\sigma^2(i)} & \sigma(i) = j \\
				-\lambda_{\sigma^2(j)}X_{\sigma^2(j)} & \sigma(j) = i \\
				0 & \text{otherwise}
			\end{cases}
		\end{equation}
	\end{definition}	
	An example of a Lie algebra with a Milnor frame of dimension $n \geq 4$ is the Lie algebra, $\mathfrak{h}^3$, associated to the 3-dimensional Heisenberg Group directly summed with an abelian Lie algebra, $\mathfrak{a}$, of dimension $n \geq 1$. For a less trivial example let $\mathfrak{h}^4$ be the $4$-dimensional Lie algebra with a frame $\{X_1,X_2,X_3,X_4\}$ whose bracket relations are generated by 
	$$
	[X_1,X_2] = X_3 \quad [X_2,X_3] = X_4.
	$$ 
	
	%Find the reference for this on Libgen and look up specific section and page number
	\noindent As we shall show in section~\ref{Examples:h3h4}, $\mathfrak{h}^4$ is a 3-step nilpotent Lie algebra and is thus not isomorphic to $\mathfrak{h}^3 \oplus \mathfrak{a}$. It turns out that any Lie algebra with a Milnor frame is completely determined by these two types of Lie algebras: 
	\begin{theorem}{\label{splitting}}
		For any Lie algebra $\mathfrak{g}$ of dimension $n \geq 4$ with a Milnor frame, $\mathfrak{g} \cong (\oplus \mathfrak{h}^3) \oplus (\oplus \mathfrak{h}^4) \oplus \mathfrak{a}$ where $\mathfrak{h}^3$ is the Lie algebra of the Heisenberg Group, $\mathfrak{h}^4$ is as seen in Example~\ref{adh4}, and $\mathfrak{a}$ is an abelian Lie Alebra. Moreover, these Lie algebras are at most 3-step nilpotent. 
	\end{theorem}
	\indent A Corollary of Theorem~\ref{splitting} shows that any $n$-dimensional, $n \geq 4$, Lie algebra with a Milnor frame admits a Milnor frame whose structure constants are either $0$ or $1$. By a result of Malcev \cite{MR0039734}, the simply connected Lie group associated to such a Lie algebra has compact quotient. If a metric Lie algebra has a Milnor frame that is also orthonormal, then we obtain some ``nice'' geometric properties. For example, given an $n$-dimensional nilpotent metric Lie algebra $(\mathfrak{g},g)$ with an orthonormal frame $\{X_1, \ldots, X_n\}$, Theorem 1.1 of \cite{MR3080187} determines an algebraic constraint for when $\{X_1, \ldots, X_n\}$ is Ricci diagonalizable. If a metric Lie algebra $(\mathfrak{g},g)$ has an orthonormal Milnor frame $\{X_1, \ldots, X_n\}$, by Theorem~\ref{splitting} and Theorem 1.1 of~\cite{MR3080187}, $\{X_1, \ldots, X_n\}$ diagonalizes the Ricci tensor. For example consider the Lie algebra $\mathfrak{h}^3 \oplus \mathfrak{a}$ with $\mathfrak{a}$ being a finite dimensional Lie algebra. Corollary 5.3 of \cite{MR1958155} shows that $\mathfrak{h}^3$ has one metric up to scaling and automorphism. Thus any Lie algebra with a Milnor frame which is isomorphic to $\mathfrak{h}^3 \oplus \mathfrak{a}$ must have an orthonormal Milnor frame. This raises the question: \\

	\begin{question}
		Given a Lie algebra $\mathfrak{g}$ that admits Milnor frame and a metric $g$ on $\mathfrak{g}$, is it possible to always find an orthonormal Milnor frame? 
	\end{question}
	\noindent Surprisingly, we find that this is not necessarily true. 
	\begin{theorem}{\label{onlyone}}
		For any nonabelian Lie algebra $\mathfrak{g}$ as in Theorem~\ref{splitting} which is not isomorphic to $\mathfrak{h}^3 \oplus \mathfrak{a}$, $\mathfrak{a}$ abelian, there exists a metric $g$ on $\mathfrak{g}$ such that $(\mathfrak{g}, g)$ does not admit an orthonormal Milnor frame. 
	\end{theorem}
	
	The paper will be divided into two sections. In Section~\ref{Algebraic Properties of Milnor Frames}, we discuss how the algebraic structure of Milnor frames can be interpreted as a directed graph. In addition we provide a proof of Theorem~\ref{splitting}. In Section~\ref{Geometric Properties}, we consider metric Lie algebras with a Milnor frame and determine necessary conditions for the Lie algebra to admit an orthonormal Milnor frame. We then provide a proof of Theorem~\ref{onlyone}. In Section~\ref{Orthonormal Milnor Frames} we show that a metric Lie algebra with an orthonormal Milnor frame has a diagonalizable Ricci tensor. Finally, we discuss which metrics on a metric Lie algebra with a Milnor frame are Ricci nilsolitons. 
	
	\section*{Acknowledgments} 
	I would like to thank Prof.~Luca Di Cerbo, those who attended my talk in the Topology and Dynamics seminar, my committee, and my peers for giving valuable feedback. In addition, I would like to thank Prof.~Michael Jablonski and the faculty at OU for giving me the opportunity to present my research for their Geometry and Topology Seminar and providing me with helpful commentary. I would like to thank the referees for providing suggestions that improve the quality of the paper.
	
	\section*{Funding}
	This paper is supported by the National Science Foundation (MS-2104662).
	\section{Algebraic Properties of Milnor Frames}\label{Algebraic Properties of Milnor Frames}
	
	\subsection{Extending Three Dimensional Milnor Frames}
	%	Recall that a $3$-dimensional unimodular Lie algebra admits a frame $\{X_1,X_2,X_3\}$ and structure constants $\lambda_1,\lambda_2,\lambda_3 \in \bb R$ where
	%	$$
	%	[X_1,X_2] = \lambda_3X_3, \quad [X_2,X_3] = \lambda_1X_1, \quad [X_3,X_1] = \lambda_2X_2.
	%	$$
	%	If $\sigma = (1 \, 2 \, 3) \in S_3$ where $S_3$ is the group of permutations acting on $3$ elements, then for any $i$ $[X_i,X_{\sigma(i)}] = \lambda_{\sigma^2(i)} X_{\sigma^2(i)}$. The goals of this section will be to construct an algebraic extension using the element $\sigma_n = (1 \, \ldots \, n)$ of the permutation group acting on $n$ elements, $S_n$, and to determine the algebraic properties of this extension. 
	%	\begin{definition}
		%		Let $\mathfrak{g}$ be a Lie algebra of dimension $n$ and $\sigma = (1 \, \ldots \, n) \in S_n$ where $S_n$ is the permutation group acting on $n$ elements. A frame, $\{X_1, \ldots, X_n\}$, is a Milnor frame if for all $i,j \in \{1, \ldots, n\}$, there exist $\lambda_1, \ldots, \lambda_n \in \bb R$ such that
		%		\begin{equation}
			%			[X_i,X_j] = \begin{cases} \lambda_{\sigma^2(i)}X_{\sigma^2(i)} & \sigma(i) = j \\
				%				-\lambda_{\sigma^2(j)}X_{\sigma^2(j)} & \sigma(j) = i \\
				%				0 & \text{otherwise}
				%			\end{cases}.
			%		\end{equation}
		%	\end{definition}	
	%	The terms $\lambda_1, \ldots, \lambda_n \in \bb R$ are referred to as the structure constants relative to the Milnor frame $\{X_1, \ldots, X_n\}$. 
	%	
	Given an $n$-dimensional, $n \geq 4$, Lie algebra $\mathfrak{g}$ with a Milnor frame, one may ask why we chose the cyclic permutation $\sigma = (1 \, \ldots \, n)$ for Definition~\ref{HDMFDef} as opposed to using any permutation in $S_n$. Suppose we change Definition~\ref{HDMFDef} so that $\sigma$ is some permutation in $S_n$. Let $\mathfrak{g}$ be a Lie algebra which admits a Milnor frame $\{X_1, \ldots, X_n\}$ with structure constants $\lambda_1, \ldots, \lambda_n$ and $\sigma \in S_n$. We can represent $\sigma = \sigma_1\cdots \sigma_k$ as a product of disjoint cycles. Let $o(i)$ denote the order of $\sigma_i$. For any element $j \in \{1, \ldots, n\}$ such that $\sigma_i(j) \neq j$, the span of $\{X_j, X_{\sigma_i(j)}, \ldots, X_{\sigma_i(o(j)-1)}\}$ will form a subalgebra $\mathfrak{g}_i$. In fact each subalgebra admits a Milnor frame $X_j, \ldots, X_{\sigma_i(o(j)-1)}$ with structure constants $\lambda_j, \ldots, \lambda_{\sigma_i(o(j)-1)}$. Thus 
	$$
	\mathfrak{g} \cong \mathfrak{g}_1 \oplus \ldots \oplus \mathfrak{g}_k \oplus \mathfrak{a}
	$$
	where $\mathfrak{a}$ is an abelian subalgebra spanned by $\{X_j \where \sigma(j) = j\}$. Indeed for each $j$ such that $\sigma(j) = j$, 
	$$
	[X_{\sigma^{-1}(j)}, X_j] = \overbrace{[X_j, X_j]}^{=0} = [X_j, X_{\sigma(j)}] 
	$$
	so that $\ad_{X_j} = 0$. This argument shows that we only need to consider the case where $\sigma$ is a cyclic permutation. 
	\begin{remark}
		For the remainder of the paper, we shall use addition modulo $n$ instead of $\sigma = (1 \, \ldots \, n) \in S_n$ to improve readability. 
	\end{remark}
	We may represent the Lie bracket between two elements of a Milnor frame in the following way: let $\delta_{ij}$ be $1$ if $i = j$ and $0$ if $i \neq j$. For a Milnor frame $\{X_1, \ldots, X_n\}$ with structure constants $\lambda_1, \ldots, \lambda_n \in \bb R$ and $i,j \in \{1, \ldots, n\}$
	\begin{equation}
		[X_i,X_j] = \delta_{(i+1)j}\lambda_{i+2}X_{i+2} - \delta_{i(j+1)}\lambda_{i+1}X_{i+1}.
	\end{equation}
	If we let $E_{ij}$ denote the matrix whose only nonzero coefficient is $1$ at the entry $(i,j)$, then for a Milnor frame $\{X_1, \ldots, X_n\}$ with structure constants $\{\lambda_1, \ldots, \lambda_n\}$, 
	\begin{equation}
		\ad_{X_i} = \lambda_{i+2}E_{(i+2)(i+1)} - \lambda_{i+1}E_{(i+1)(i-1)}. \label{adjoint}
	\end{equation}
	
	\subsection{Examples: $\mathfrak{h}^3$ and $\mathfrak{h}^4$}\label{Examples:h3h4}
	We start with the classic example of a nilpotent Lie algebra with a Milnor frame. 
	\begin{example}
		\normalfont The 3-dimensional Heisenberg group is a subgroup of $GL(3, \bb R)$ whose elements are of the form 
		\[
		\mathcal{H} = \left\{\begin{pmatrix}1 & a & b \\ 0 & 1& c \\ 0 & 0 &1 \end{pmatrix} {\bigg|} a,b,c \in \bb R\right\}.
		\]
		The Lie algebra of $\mathcal{H}$ is
		$$
		\mathfrak{h}^3 = \left\{\begin{pmatrix}0 & a & b \\ 0 & 0 & c \\ 0 & 0 & 0  \end{pmatrix} {\bigg|} a,b,c \in \bb R\right\}.
		$$
		Consider the frame $X_1 := E_{12}, X_2:=E_{13}$ and $X_3:= E_{23}$. We obtain the bracket relations $[X_1,X_2] = X_3$ and all other bracket relations trivial. If $\mathfrak{a}$ is an abelian Lie algebra of dimension $n \geq 1$, then $\mathfrak{h}^3 \oplus \mathfrak{a}$ forms an $(n+3)$-dimensional Lie algebra with a Milnor frame.
	\end{example} 
	\begin{example}{\label{adh4}}
		\normalfont Let $\mathfrak{h}^4$ be the 4-dimensional Lie algebra with a frame $\{X_1, X_2, X_3, X_4\}$ and bracket relations 
		$$
		[X_1,X_2] = X_3, \quad  [X_2,X_3] = X_4.
		$$
		The Lie algebra $\mathfrak{h}^4$ is 3-step nilpotent as shown below: 
		\begin{alignat}{1}
			&(\mathfrak{h^4})^1 = [\mathfrak{h^4},\mathfrak{h^4}] = \text{Span}\{X_3,X_4\} \notag \\
			&(\mathfrak{h^4})^2 = [\mathfrak{h^4}, (\mathfrak{h^4})^1] = \text{Span}\{X_4\} \notag \\
			&(\mathfrak{h^4})^3 = [\mathfrak{h^4}, (\mathfrak{h^4})^2] = \{0\}
		\end{alignat} 
	\end{example}
	Because every 2-dimensional Lie algebra is either abelian or non-nilpotent, $\mathfrak{h}^4$ cannot be split into two 2-dimensional Lie algebras. The Lie algebra $\mathfrak{h}^3$ is at most 2-step nilpotent as shown in~(\ref{2step})
	\begin{alignat}{1}
		&(\mathfrak{h^3})^1 = [\mathfrak{h^3}, \mathfrak{h^3}] = \text{Span}\{X_3\} \notag \\
		&(\mathfrak{h^3})^2 = [\mathfrak{h^3}, (\mathfrak{h^3})^1] = \{0\}. \label{2step}
	\end{alignat}
	Thus $\mathfrak{h}^4$ is not isomorphic to $\mathfrak{h}^3 \oplus \mathfrak{a}$.\\
	
	%Set $F_1 = X_4$, $F_2 = X_3$, $F_3 = X_1$ and $F_4 = X_2$. The nontrivial structure constants with respect to the frame $\{F_1,\ldots, F_4\}$ are $C_{2,4}^1 = -1$, $C_{3,4}^2 = 1$. By Lemma 3 of~\cite{MR2655935}, $\mathfrak{h}^4$ is the unique 3-step Nilpotent Lie algebra of dimension 4. 
	Given a Lie algebra of dimension $n \geq 4$ with a Milnor frame the upper bound for the total number of nontrivial structure constants is at most $n$. The next proposition shows that this upper bound cannot be achieved. 
	\begin{proposition}{\label{Nilpotent}}
		Let $\mathfrak{g}$ be an $n$-dimensional Lie algebra, $n \geq 4$, with a Milnor frame $\{X_1, \ldots, X_n\}$ and structure constants $\lambda_1, \ldots, \lambda_n$. Then $\lambda_i\lambda_{i+2} = 0$ for any $i \in \{1, \ldots, n\}$. 
	\end{proposition}
	
	\begin{proof}
		For any $i \in \{1, \ldots, n\}$, 
		\begin{alignat}{1}
			&[X_{i+3}, [X_i,X_{i+1}]] = \lambda_{i+2}[X_{i+3}, X_{i+2}] = \lambda_{i+2}\lambda_{i+4}X_{i+4} \notag \\
			&[X_i,[X_{i+1},X_{i+3}]] = [X_i,0] = 0 \notag \\
			&[X_{i+1},[X_{i+3},X_i]] = \begin{cases}[X_{i+1}, 0] = 0 & n > 4 \\
				\lambda_{i+1}[X_{i+1},X_{i+1}]  = 0 & n = 4\end{cases} \notag 
		\end{alignat}
		By the Jacobi Identity
		$$
		0 = [X_{i+3},[X_i,X_{i+1}]] + [X_i,[X_{i+1},X_{i+3}]] + [X_{i+1},[X_{i+3},X_i]] = \lambda_{i+2}\lambda_{i+4}X_{i+4}.
		$$
		Considering that $i$ was arbitrary, we obtain the desired identity. 
	\end{proof}
	Before we prove Theorem~\ref{splitting} we prove the following lemmas.
	\begin{lemma}{\label{scaleh3}}
		Let $\mathfrak{g}$ be a 3 dimensional Lie algebra with a Milnor frame with exactly one nontrivial structure constant. Then $\mathfrak{g} \cong \mathfrak{h}^3$ 
	\end{lemma}
	
	\begin{proof}
		Let $\{X_1, X_2,X_3\}$ be a Milnor frame with structure constants $\lambda_1, \lambda_2, \lambda_3$. Suppose $\lambda_k \neq 0$ for some $k$. If $\mathfrak{g}$ is defined through the bracket relation $[X_i,X_j] = \lambda_k X_k$ for some $i,j \neq k$ then by setting $Y_1 = X_i$, $Y_2 = X_j$ and $Y_3 = X_k$, we obtain $[Y_1,Y_2] = \lambda_3Y_3$. By scaling $Y_3$ to $\lambda_3Y_3$, $\{Y_1, Y_2, \lambda_3Y_3\}$ forms a Milnor frame whose structure constants are either $0$ or $1$. This shows that $\mathfrak{g} \cong \mathfrak{h}^3$.
	\end{proof}
	\begin{lemma}{\label{scaleh4}}
		Let $\mathfrak{g}$ be a 4-dimensional Lie algebra with a Milnor frame with exactly two nontrivial structure constants. Then $\mathfrak{g} \cong \mathfrak{h}^4$. 
		
		\begin{proof}
			Let $\{X_1,\ldots, X_4\}$ be a Milnor frame with structure constants $\lambda_1, \ldots, \lambda_4$. Suppose $\lambda_k,\lambda_\ell \neq 0$ are our two nontrivial structure constants. If $|k - \ell| = 2$ by Proposition~\ref{Nilpotent}, $\lambda_k\lambda_\ell = \lambda_j\lambda_{j+2} = 0$ for some $j$ which is a contradiction. Thus we may assume, without loss of generality, that $\ell = k +1$. By relabeling and rescaling, we obtain that the frame $$\{Y_1 = X_{k-2}, Y_2 = X_{k-1}, Y_3 = \lambda_kX_k, Y_4 = \frac{\lambda_{k+1}}{\lambda_k} X_{k+1}\}$$ is a Milnor frame whose structure constants are either $0$ or $1$. This gives the relation $\mathfrak{g} \cong \mathfrak{h}^4$. 
		\end{proof}
	\end{lemma}	
	
	\noindent Now for the proof of Theorem~\ref{splitting}.
	\begin{proof}[Proof of Theorem~\ref{splitting} \nopunct]\label{proof of splitting}
		First we want to show that there exists $i \in \{1, \ldots, n\}$ such that $\lambda_i \neq 0$ and $\lambda_{i-1} = 0$. Suppose $\lambda_i \neq 0$ for some $i$ and additionally suppose $\lambda_{i-1} \neq 0$. By Proposition~\ref{Nilpotent} $\lambda_{i-2}\lambda_i = 0$ which shows $\lambda_{i-2} = 0$. 
		
		\smallskip 
		We proceed by induction on the dimension of the Lie algebra. Assume that any Lie algebra of dimension $k < n$ with a Milnor frame splits as a direct sum of ideals as stated in the theorem. Let $\{X_1, \ldots, X_n\}$, $n \geq 4$, be a Milnor frame for the nonabelian Lie algebra $\mathfrak{g}$. Noting that the operator $S_\ell: \mathfrak{g} \to \mathfrak{g}$, defined by $c^iX_i \mapsto c^i X_{i+\ell}$, takes Milnor frames to Milnor frames, by our previous comments we may suppose without loss of generality that $\lambda_3 \neq 0$ and $\lambda_2 = 0$. We consider the cases where $\lambda_4 = 0$ and $\lambda_4 \neq 0$. 
		
		\smallskip
		Suppose $\lambda_4 = 0$. For each $i$, $$
		\ad_{X_i}(\mathfrak{g}) = \Span\{\lambda_{i+1}X_{i+1}, \lambda_{i+2}X_{i+2}\} \subset \Span\{X_{i+1}, X_{i+2}\}.
		$$
		Thus, in order to determine if $\{X_1, X_2,X_3\}$ is an ideal we need to show that $\ad_{X_3} = 0$. By our assumption,  $\lambda_3 \neq 0$ so that by Proposition~\ref{Nilpotent}, $\lambda_3\lambda_5 = 0$ which implies $\lambda_5 = 0$. Given that $\lambda_4 = 0$ this shows $\ad_{X_3} = 0$. By Lemma~\ref{scaleh3} and the fact that $\lambda_1,\lambda_2 = 0$, $\Span\{X_1,X_2,X_3\} \cong \mathfrak{h}^3$. To show that $\Span\{X_4, \ldots, X_n\}$ forms an ideal we must show $[X_3,X_4] = 0$ and $\ad_{X_n} = 0$. We know $\ad_{X_3} = 0$ so that $[X_3,X_4] = 0$. By assumption, $\lambda_2 = 0$ and by Proposition~\ref{Nilpotent} $\lambda_3\lambda_1 = 0$ which implies $\lambda_1 = 0$. Thus $\ad_{X_n} = 0$ and so $\Span\{X_4, \ldots, X_n\}$ forms an ideal.
		
		\smallskip
		Now suppose that $\lambda_4 \neq 0$. We show that $\Span\{X_1,\ldots, X_4\}$ and $\Span\{X_5, \ldots, X_n\}$ are ideals. Similar to the previous argument, we show that $\ad_{X_4} = 0$ and $\ad_{X_n} = 0$. Due to $\lambda_3,\lambda_4 \neq 0$ and Proposition~\ref{Nilpotent} we can determine that $\lambda_5, \lambda_6 = 0$ which shows $\ad_{X_4}= 0$. Because $\lambda_1,\lambda_2 = 0$, by Lemma~\ref{scaleh4}, $\Span\{X_1, \ldots, X_4\} \cong \mathfrak{h}^4$. Now we show that $\ad_{X_n} =0$. By assumption $\lambda_2 = 0$. We know $\lambda_1 = 0$ by Proposition~\ref{Nilpotent} and the fact that $\lambda_3 \neq 0$ so that $\ad_{X_n}= 0$. This shows that $\Span\{X_1, \ldots, X_4\}$ and $\Span\{X_5, \ldots, X_n\}$ form ideals.
		
		\medskip
		Thus $\mathfrak{g}$ is isomorphic to $\mathfrak{h}^3 \oplus \mathfrak{h}$ or $\mathfrak{h}^4 \oplus \mathfrak{h}$ where $\mathfrak{h}$ is an ideal of $\mathfrak{g}$. To complete the proof we must show that $\mathfrak{h}$ is an ideal with a Milnor frame. Let $m \in \{3,4\}$ so that $\mathfrak{h} = \Span\{X_{m+1}, \ldots, X_n\}$. We claim that if $n-m \leq 2$ then $\mathfrak{h}$ must be abelian. If $n - m = 2$ then 
		$$
		\ad_{X_{m+1}} = \lambda_{m+3}E_{(m+3)(m+2)} - \lambda_{m+2}E_{(m+2)m} = \lambda_1 E_{1m} = 0
		$$
		where the last two equalities follow from $m+3 \mod n = n + 1 \mod n = 1 \mod n$, Proposition~\ref{Nilpotent}, and $\lambda_3,\lambda_m \neq 0$. Furthermore, we have already shown that in both cases $\ad_{X_{m+2}} = \ad_{X_n} = 0$ so that $\mathfrak{h}$ is abelian. If $n - m = 1$ then $\mathfrak{h}$ is a one dimensional ideal of a Lie algebra and is thus abelian. 
		Suppose that $n - m > 2$. For each $m+1 \leq j \leq n-2$ $\ad_{X_j} = \lambda_{j+2}E_{(j+2)(j+1)} - \lambda_{j+1}E_{(j+1)(j-1)}$. By previous observations, $\ad_{X_n} = 0$. Thus we have that $[X_{n-1}, X_n] = 0\cdot X_m$ and $[X_n, X_m] = 0\cdot X_{m+1}$ so that $\mathfrak{h}$ admits a Milnor frame. By induction, $\mathfrak{h}$ splits into a sum of ideals as stated in Theorem~\ref{splitting}.
	\end{proof}
	\section{Geometric Properties of Milnor Frames}\label{Geometric Properties}
	Let $(\mathfrak{g},g)$ be a metric Lie algebra with a Milnor frame. By Theorem~\ref{splitting} we know that $\mathfrak{g} \cong (\oplus \mathfrak{h}^3) \oplus (\oplus \mathfrak{h}^4) \oplus \mathfrak{a}$. Let us assume that such a decomposition of $\mathfrak{g}$ is pair-wise orthogonal. Each $\mathfrak{h}^3$ is a 3-dimensional unimodular metric Lie algebra, and so by Lemma 4.1 of~\cite{MR425012} there exists an orthonormal Milnor frame of $\mathfrak{h}^3$. Thus in order for $\mathfrak{g}$ to have an orthonormal Milnor frame, the subalgebras in the decomposition of $\mathfrak{g}$ isomorphic to $\mathfrak{h}^4$ must have an orthonormal Milnor frame. We now determine precisely when $\mathfrak{h}^4$ admits an orthonormal Milnor frame.

	\subsection{Metrics on $\mathfrak{h}^4$}
	Consider the 4-dimensional Lie algebra $A_{4,1}$ as seen on page 246 of~\cite{MR2655935} whose bracket relations are shown below:
	$$
	[X_2,X_4] = X_1, \quad [X_3,X_4] = X_2.
	$$
	By setting $Y_1 = X_2 + X_3$, $Y_2 = X_4$, $Y_3 = X_1 + X_2$, $Y_4 = -X_1$ we obtain $ A_{4,1} \cong \mathfrak{h}^4$ as shown in Figure~\ref{a41}.

	\medskip
	\begin{figure}[ht]
		\begin{center}
			\begin{tabular}{|c|c|c|c|c|}
				\hline
				& $Y_1$ & $Y_2$& $Y_3$ & $Y_4$ \\
				\hline
				$Y_1$ & 0 & $X_1 + X_2$ & 0 & 0 \\
				\hline
				$Y_2$ & $-X_1 - X_2$ &0& $-X_1$ & 0\\
				\hline
				$Y_3$ & $0$ & $X_1$ & 0 & 0 \\
				\hline 
				$Y_4$ & 0 & 0 & 0 & 0 \\
				\hline
			\end{tabular}
		\end{center}
		\caption{Entry $i,j$ as $[Y_i,Y_j]$}
		\label{a41}
	\end{figure}
	
	\noindent Thus in Lemma 3 of~\cite{MR2655935}, there exists an orthonormal frame $\{F_1, \ldots, F_4\}$ of $\mathfrak{h}^4$ for which there are at most three nontrivial structure constants $C_{2,4}^1,C_{3,4}^2 > 0$, $C_{3,4}^1 \in \bb R$. We now show that $\mathfrak{h}^4$ has an orthonormal Milnor frame if and only if $C_{3,4}^1 = 0$. 
	\begin{theorem}{\label{orthonormalh4}}
		Let $\{F_1, F_2, F_3,F_4\}$ be an orthonormal frame of the metric Lie algebra $\mathfrak{g}$ generated by the following Lie bracket relations: $a = C_{2,4}^1> 0$, $b = C_{3,4}^1 \in \bb R$, and $c= C_{3,4}^2 > 0$. Let $O(4)$ be the orthogonal group of $4\times 4$ matrices with real coefficients. There exists $T \in O(4)$ which makes $\{TF_1,TF_2,TF_3,TF_4\}$ into a Milnor frame with 2 nontrivial structure constants if and only if $b = 0$. 
	\end{theorem}
	
	\begin{proof}
		We prove the above proposition by contradiction. Let $T \in O(4)$ be a linear operator such that $\{TF_1,TF_2,TF_3,TF_4\}$ forms a Milnor frame. By Lemma~\ref{scaleh4} we may assume that the nontrivial structure constants of this Milnor frame are $\lambda_3, \lambda_4 \neq 0$ and $T \in O(4)$. For any $i,j$ such that $[TF_i,TF_j] = 0$ 
		\begin{alignat}{1}
			0 &= [TF_i,TF_j] = [T_i^pF_p,T_j^qF_q] = T_i^pT_j^q[F_p,F_q] \notag \\
			&= \left[a(T_i^2T_j^4 - T_i^4T_j^2) + b(T_i^3T_j^4 - T_i^4T_j^3)\right]F_1 + c(T_i^3T_j^4 - T_i^4T_j^3)F_2. \notag 
		\end{alignat}  
		As a result $T_i^3T_j^4 = T_i^4T_j^3$ and $T_i^2T_j^4 = T_i^4T_j^2$. By the fact that $[\mathfrak{g},\mathfrak{g}] \in \spa\{F_1,F_2\}$,
		\begin{alignat}{1}
			&[TF_1,TF_2] = \lambda_3TF_3 \implies T_3^3 = T_3^4 = 0 \notag \\
			&[TF_2,TF_3] = \lambda_4TF_4 \implies T_4^3 = T_4^4 = 0. \notag
		\end{alignat} 
		
		\medskip
		We claim that $T_1^4$ is necessarily trivial. Suppose otherwise. Because $T_3^4 = 0$ and $[TF_1,TF_3] = [TF_1,TF_4] = 0$, 
		$0 = T_1^2T_3^4 = T_1^4T_3^2 \implies T_3^2 = 0$. Similarly $0 = T_1^2T_4^4 = T_1^4T_4^2 \implies T_4^2 = 0$ so that the vectors $[T_1^4 \ldots T_4^4]^t$ and $[T_1^3, \ldots , T_4^3]^t$ are linearly dependent, a contradiction to $T \in O(4)$. 
		
		\medskip
		Because $T_j^4 = 0$ for $j = 1,3,4$ and $T \in O(4)$, $T^4_2T_2^j=\sum_{k=1}^4 T^4_kT_k^j = \delta_{4j}$. We obtain that $T^4_2 = \pm 1$ which implies $T_2^1 = T_2^2 = T_2^3 = 0$. Similarly $T_1^3T_1^j = \sum_k T_k^3T_k^j = \delta_{3j}$ implies $T_1^3 = \pm 1$ and $T_1^1 = T_1^2 = T_1^4 = 0$. The matrix representation of $T$ under the basis $(F_1,\ldots, F_4)$ is given by 
		$$
		T = 
		\begin{bmatrix}
			0 & 0 & \gamma & \epsilon \\ 
			0 & 0 & \delta & \iota \\
			\alpha & 0 & 0 & 0 \\
			0 & \beta & 0 & 0
		\end{bmatrix}
		$$
		where $\alpha,\beta \in \{-1,1\}$ and $\gamma,\delta, \epsilon, \iota \in \bb R$. Because $[TF_2,TF_3] = \lambda_4TF_4$ and $\ad_{F_1} = 0$, $\iota = 0$. In order for $0 = g(TF_3,TF_4) = \gamma\epsilon$, $\gamma = 0$. Finally $[TF_1,TF_2] = \alpha\beta[F_3,F_4] = \alpha\beta(bF_1 + cF_2) = \lambda_3TF_3 = \lambda_3\delta F_2$. Considering that $\alpha\beta \in \{-1,1\}$, it must be the case that the linear operator $T$ turns the frame $\{F_1,F_2,F_3,F_4\}$ into an orthonormal Milnor frame if and only if $b = 0$. 
	\end{proof}
	
	As a result of Theorem~\ref{orthonormalh4} we can construct a metric $g$ so that the metric Lie algebra $(\mathfrak{h}^4, g)$ does not admit an orthonormal Milnor frame. If $(\mathfrak{h}^4,g)$ has an orthonormal Milnor frame $\{X_1,\ldots, X_4\}$, then there exists $T \in O(4)$ such that $\{TX_1, \ldots, TX_4\} = \{F_1, \ldots, F_4\}$ where $\{F_1, \ldots, F_4\}$ is as in~(\ref{orthonormalh4}). Thus $\{T^{-1}F_1, \ldots, T^{-1}F_4\}$ is an orthonormal Milnor frame which implies that $g([F_3,F_4],F_1) = 0$. An example of a metric $g$ such that $g([F_3,F_4],F_1) \neq 0$ is found in Lemma 3 of~\cite{MR2655935}. 
	
	\subsubsection{Metrics on $\mathfrak{h}^3 \oplus \mathfrak{h}^3$}
	Now we consider the collection of metric Lie algebras $\mathfrak{g}$ which contain an isomorphic copy of $\mathfrak{h}^3 \oplus \mathfrak{h}^3$. If we are able to provide a metric that does not allow $\mathfrak{h}^3 \oplus \mathfrak{h}^3$ to have an orthogonal Milnor frame, then we obtain Theorem~\ref{onlyone} as a result. First we prove the following result from linear algebra.

	\begin{lemma}\label{linlemma}
		Let $A,B \in M_2(\bb R)$. Denote the $i^{th}$ column of an arbitrary matrix $C \in M_2(\bb R)$ as $C_i$. If $A$ and $B$ have the property $A_i$ and $B_j$ are linearly dependent for any $i,j \in \{1,2\}$ then either $A$ or $B$ has determinant $0$. 
	\end{lemma}
	
	\begin{proof}
		Suppose that $A$ is non-singular. By our hypothesis $B_i = c_{ij}A_j$ for some $c_{ij} \in \bb R$ and $i,j \in \{1,2\}$. Thus 
		$$
		\begin{bmatrix}A_1 & A_2\end{bmatrix}\begin{bmatrix} 0 & c_{21} \\ c_{12} & 0  \end{bmatrix} =	\begin{bmatrix}B_1 & B_2\end{bmatrix} = \begin{bmatrix}A_1 & A_2 \end{bmatrix}\begin{bmatrix}c_{11} & 0 \\ 0 & c_{22} \end{bmatrix}.
		$$
		Because $A$ is invertible, $c_{ij} = 0$ for all $i,j$ which implies $B_1 = B_2 = 0$. 
	\end{proof}
	
	\begin{proposition}{\label{nonorthogonalh3h3}}
		Let $(\mathfrak{h}^3 \oplus \mathfrak{h}^3, g)$ be a metric Lie algebra. Consider the Milnor frame $\{U_1,U_2,U_3,V_1,V_2,V_3\}$ for $(\mathfrak{h}^3 \oplus \mathfrak{h}^3)$ where $$\mathfrak{h}^3 \oplus \{0\} = \text{Span}\{U_1,U_2,U_3\} \quad \text{and} \quad \{0\} \oplus \mathfrak{h}^3 = \text{Span}\{V_1,V_2,V_3\}.$$ If there exists a Lie isomorphism $T$ with $g(T(\mathfrak{h}^3\oplus \{0\}), T(\{0\} \oplus \mathfrak{h}^3)) = 0$ then $g(U_3,V_3) = 0$.
	\end{proposition}
	
	\begin{proof}
		We may always scale the frames $\{U_1,U_2,U_3\}$ and $\{V_1,V_2,V_3\}$ so that the structure constants are contained in $\{0,1\}$. Without loss of generality we shall assume that the structure constants for the frame $\{U_1,\ldots, V_3\}$ are contained in $\{0,1\}$. Because $T$ is a Lie isomorphism, $T_1^iT_2^j[U_i,U_j] = [TU_1,TU_2] = TU_3 = T^k_3U_k$. This shows $T_3^3 = T_1^1T_2^2 - T_1^2T_2^1$, $T_3^6 = T_1^4T_2^5- T_1^5T_2^4$, and $T_3^k = 0$ for $k = 1,2,4,5$. Similarly $T_6^3 = T_4^1T_5^2-T_4^2T_5^1$, $T_6^6 = T_4^4T_5^5 - T_4^5-T_5^4$ and $T_6^k = 0$ for $k=1,2,4,5$. For $i,j$ such that $[P_i,P_j] = 0$, $P_i,P_j \in \{U_1, \ldots, V_3\}$, $T_i^1T_j^2 - T_i^2T_j^1 = T_i^4T_j^5 - T_i^5T_j^4 = 0$. The matrix representation of $T$ can be given by the block matrix 
		$$
		T:= \begin{bmatrix} W & 0 & X & 0 \\ w^t & \det(W) & x^t & \det(X) \\ Y & 0 & Z & 0 \\y^t & \det(Y)  & z^t & \det(Z) \end{bmatrix}
		$$
		Where $W,X,Y$ and $Z$ are $2\times 2$ matrices, $w,x,y,z$ are $2\times 1$ matrices, and each $0$ is a $2 \times 1$ matrix. Letting $C_i$ represent the $i^{th}$ column of a matrix $C \in M_2(\bb R)$, we have the additional property that $W_i$ and $X_j$ are linearly dependent for any $i,j$ and similarly $Y_i$ and $Z_j$ are linearly dependent. Suppose that $X$ is singular. If $W$ or $Z$ is singular then $T$ will be singular leading us to a contradiction. Let $W$ and $Z$ be nonsingular. By Lemma~\ref{linlemma} $Y$ is singular.  $T$ can be represented as 
		$$
		T = \begin{bmatrix} \omega & \chi \\ \gamma & \zeta \end{bmatrix}
		$$
		where 
		\begin{alignat}{2}
			&\omega = \begin{bmatrix} W & 0 \\ w^t & \det(W) \end{bmatrix} \hspace{2cm} &&\chi = \begin{bmatrix} X & 0 \\ x^t & 0\end{bmatrix} \notag \\
			&\gamma = \begin{bmatrix} Y & 0 \\ y^t & 0 \end{bmatrix} && \zeta = \begin{bmatrix} Z & 0 \\ z^t & \det(Z) \end{bmatrix}. \notag
		\end{alignat}
		Letting $g$ be represented as 
		$$
		g = \begin{bmatrix}A & C \\ C^t & B \end{bmatrix}
		$$
		\begin{alignat}{1}
			&\begin{bmatrix} \omega^t & \gamma^t \\ \chi^t & \zeta^t\end{bmatrix}\begin{bmatrix}A & C \\ C^t & B\end{bmatrix}\begin{bmatrix}\omega & \chi \\ \gamma & \zeta\end{bmatrix} = \begin{bmatrix}\omega^t A + \gamma^t C^t & \omega^t C + \gamma^t B \\ \chi^t A + \zeta^t C^t & \chi^t C + \zeta^t B\end{bmatrix}\begin{bmatrix}\omega & \chi \\ \gamma & \zeta \end{bmatrix} \notag \\
			&= \begin{bmatrix}\omega^tA \omega + \gamma^tC^t \omega + \omega^t C\gamma + \gamma^t B \gamma & \omega^t A \chi + \gamma^tC^t \chi + \omega^tC\zeta + \gamma^t B \zeta \\ \chi^tA\omega+ \zeta^tC^t\omega + \chi^t C \gamma + \zeta^t B \gamma & \chi^tA\chi + \zeta^tC^t \chi+ \chi^tC\zeta + \zeta^t B \zeta \end{bmatrix}
		\end{alignat}
		so that 
		\begin{equation}
			\omega^tA\chi + \gamma^tC^t \chi + \omega^tC\zeta + \gamma^t B \zeta = 0.
		\end{equation}
		Finally we obtain 
		\begin{alignat}{1}
			0 &= (\omega U_3)^tA (\chi V_3) + (\gamma U_3)^t C^t (\chi V_3) + (\omega U_3)C(\zeta V_3) + (\gamma U_3)^t B (\zeta V_3) \notag \\
			&= \det(W)\det(Z)(U_3)^t C V_3 = \det(W)\det(Z)g(U_3,V_3) 
		\end{alignat}
		where $\det(W)\det(Z) \neq 0$ so that $g(U_3,V_3) = 0$. 
		
		If $X$ is non-singular then by Lemma~\ref{linlemma} $W$ and $Z$ must be singular. In order for $T$ to remain non-singular $Y$ must be nonsingular. If we post-compose $T$ with the Lie isomorphism $S = \begin{bmatrix}0 & I \\ I & 0\end{bmatrix}$ then a similar argument of the case where $X$ is singular holds to show that $g(U_3,V_3) = 0$. 
	\end{proof}
	Taking the contrapositive of the previous theorem states if $g(U_3,V_3) \neq 0$, then there are no linear operators $T \in \text{Aut}(\mathfrak{g})$ such that $g(TU_i,TV_j) = 0$ $\forall i,j$. 
	
	\medskip
	\begin{proposition}{\label{orthonormalh3h3}}
		There exists a metric $g$ on the Lie algebra $(\mathfrak{h}^3 \oplus \mathfrak{h}^3)$ such that $(\mathfrak{h}^3 \oplus \mathfrak{h}^3, g)$ does not admit an orthonormal Milnor frame. 
	\end{proposition}
	
	\begin{proof}
		Let be $(\mathfrak{g} = \mathfrak{h}^3 \oplus \mathfrak{h}^3,g)$ be a metric Lie algebra with a Milnor frame $\{X_1, \ldots, X_6\}$ and $T: \mathfrak{g} \to \mathfrak{g}$ be a linear operator such that $\{TX_1, \ldots, TX_6\}$ forms a Milnor frame with two nontrivial structure constants. Denote these structure constants as $\lambda_1, \ldots, \lambda_6 \in \{0,1\}$. Then either $\lambda_i = \lambda_{i+1} = 1$, $\lambda_i = \lambda_{i+2} = 1$, or $\lambda_i = \lambda_{i+3} =1 $ for some $i \in \{1,\ldots, 6\}$. It cannot be the case that $\lambda_i = \lambda_{i+2} = 1$ for some $i$ by Proposition~\ref{Nilpotent} which leaves us with two cases. 
		
		\medskip 
		If there is some $i$ such that $\lambda_i = \lambda_{i+1}$, then $\mathfrak{h}^4 \subset \mathfrak{h}^3 \oplus \mathfrak{h}^3$. By Proposition~\ref{orthonormalh4}, we may choose a metric $g$ so that $(\mathfrak{h}^4, g\vert_{\mathfrak{h}^4})$ does not admit an orthonormal Milnor frame and so $(\mathfrak{h}^3 \oplus \mathfrak{h}^3, g)$ does not admit an orthonormal Milnor frame.
		
		\medskip 
		If $\lambda_i = \lambda_{i+3} = 1$ for some $i$, then $T \in \text{Aut}(\mathfrak{g})$. Choose a metric $g$ which makes $g(X_3,X_6) \neq 0$. For example if $g = I + \epsilon(E_{36} + E_{63})$ for some sufficiently small $\epsilon > 0$ that allows $g$ to be positive definite then $g(X_3,X_6) = \epsilon > 0$. By Proposition~\ref{nonorthogonalh3h3}, there is no linear operator $T$ which makes $T(\mathfrak{h}^3 \oplus \{0\})$ the orthogonal complement of $T(\{0\}\oplus \mathfrak{h}^3)$. Thus $\mathfrak{h}^3 \oplus \mathfrak{h}^3$ does not necessarily admit an orthonormal Milnor frame. 
	\end{proof}
	
	Now we prove Theorem~\ref{onlyone}. 
	\begin{proof}
		Let $\mathfrak{g}$ be a Lie algebra with a Milnor frame. Suppose that any metric $g$ admits an orthonormal Milnor frame. Then $\mathfrak{h}^4$ cannot be a subalgebra of $\mathfrak{g}$ and $\mathfrak{h}^3 \oplus \mathfrak{h}^3$ cannot be a subalgebra of $\mathfrak{g}$. By Theorem~\ref{splitting}, $\mathfrak{g} = \mathfrak{h}^3 \oplus \mathfrak{a}$. 
	\end{proof}
	
	\section{Orthonormal Milnor Frames and Ricci Solitons}\label{Orthonormal Milnor Frames}
	In this section we present further results about metric Lie algebras which admit orthonormal Milnor frames and Ricci soliton metrics. In particular, we show that Ricci solitons on $\mathfrak{h}^4$ are contained in the collection of metrics $g$ on $\mathfrak{h}^4$ which admit an orthonormal Milnor frame. Recall, because of Theorem~\ref{orthonormalh4}, we can find a metric $g$ such that $(\mathfrak{h}^4,g)$ does not admit an orthonormal Milnor Frame. 
	\subsection{Ricci Tensors}
	Suppose we have a metric Lie algebra $(\mathfrak{g},g)$ with an orthonormal Milnor frame. By Theorem~\ref{splitting} the Lie algebra $\mathfrak{g}$ is isomorphic to $(\oplus \mathfrak{h}^3) \oplus(\oplus \mathfrak{h}^4) \oplus \mathfrak{a}$ where $\mathfrak{a}$ is an abelian Lie algebra. Due to $\mathfrak{g}$ admitting an orthonormal Milnor frame, the metric $g$ can be represented as 
	$$
	g = (\oplus g\vert_{\mathfrak{h}^3}) \oplus (\oplus g\vert_{\mathfrak{h}^4}) \oplus (\oplus g\vert_{\mathfrak{a}})
	$$
	and so
	$$
	\ric_g \vert_{\mathfrak{g}} = (\oplus \ric_g\vert_{\mathfrak{h}^4}) \oplus (\oplus \ric_g \vert_{\mathfrak{h}^4}) \oplus (\ric_g \vert_{\mathfrak{a}}).
	$$
	With this in mind, to determine information about the Ricci tensor for $(\mathfrak{g},g)$, it suffces to determine information about the Ricci tensors for $(\mathfrak{h}^3, g\vert_{\mathfrak{h}^3})$ and $(\mathfrak{h}^4, g\vert_{\mathfrak{h}^4})$. \\ 
	
	Let $(\mathfrak{h}^3, g)$ have an orthonormal Milnor frame whose nontrivial structure constant is $\lambda_3$. A result of~\cite[pg. 305]{MR425012} shows that 
	\begin{equation}
		\ric_g\vert_{\mathfrak{h}^3} = \frac{\lambda_3^2}{2}\begin{bmatrix}-1 & 0 & 0 \\ 0 & -1& 0 \\ 0 & 0 & 1 \end{bmatrix}. \label{Ric3matrix}
	\end{equation}
	Let $\{X_1, \ldots, X_4\}$ be an orthonormal Milnor frame for $\mathfrak{h}^4$ with nontrivial structure constants $\lambda_3$ and $\lambda_4$. The sectional curvatures can be computed using the formula found in Lemma 1.1 of~\cite{MR425012}. The sectional curvature between any two elements of the Milnor frame can be found in Figure~\ref{SCTable}. 
	\begin{figure}[ht]
		\begin{center}
			\begin{tabular}{|c|c|c|c|c|}
				\hline
				& $X_1$ & $X_2$&$X_3$ & $X_4$ \\
				\hline
				$X_1$ & 0 & $-\frac{3}{4}\lambda_3^2$ & $\frac{1}{4}\lambda_3^2$ & 0\\
				\hline 
				$X_2$ & $-\frac{3}{4}\lambda_3^2$ & 0 & $-\frac{3}{4}\lambda_4^2 + \frac{1}{4}\lambda_3^2$ & $\frac{1}{4}\lambda_4^2$\\
				\hline 
				$X_3$ & $\frac{1}{4}\lambda_3^2$ & $-\frac{3}{4} \lambda_4^2 + \frac{1}{4}\lambda_3^2$ & 0 & $\frac{1}{4}\lambda_4^2$ \\
				\hline
				$X_4$ &$0$ &$\frac{1}{4}\lambda_4^2$ & $\frac{1}{4}\lambda_4^2$ & 0 \\
				\hline
			\end{tabular}
		\end{center}
		\caption{$\kappa(X_i,X_j)$ := $i,j$\textsuperscript{th} entry.}
		\label{SCTable}
	\end{figure} \\
	
	\vspace{.5cm}
	\noindent By Theorem 1.1, Lauret and Will~\cite[pg. 3652]{MR3080187} $\ric_g(X_i,X_j) = 0$ for $i \neq j$.  By summing each row of Figure~\ref{SCTable} we obtain the matrix representation of the Ricci tensor: 
	\begin{equation}{\label{Ric4matrix}}
		\ric_g\vert_{\mathfrak{h}^4} = 
		\frac{1}{2}\begin{bmatrix}
			-\lambda_3^2 & 0 & 0 & 0 \\
			0 & -\lambda_3^2 - \lambda_4^2 & 0 & 0 \\
			0 & 0 & \lambda_3^2 -\lambda_4^2& 0 \\
			0 & 0  & 0 & \lambda_4^2 
		\end{bmatrix}.
	\end{equation}

	\medskip
	The eigenvalues of the Ricci signatures for $\ric_g\vert_{\mathfrak{h}^3}$ and $\ric_g\vert_{\mathfrak{h}^4}$ admit positive and negative terms which coincides with Theorem 2.4 of~\cite[pg. 301]{MR425012}. An immediate consequence of this theorem is that any non-commutative nilpotent metric Lie algebra does not admit an Einstein metric, a metric $g$ for which $\text{Ric}_g = \lambda g$ for some $\lambda \in \bb R$. A known extension of Einstein metrics are metrics which satisfy the \textit{Ricci soliton equation}. For a Riemannian manifold $(M,g)$, the metric $g$ is a Ricci soliton if $g$ satisfies the equation
	\begin{equation}
		-2\ric_g = \lambda g + \mathcal{L}_X g, \quad \lambda \in \bb R \label{ricsol}
	\end{equation}
	where $X$ is a vector field and $\mathcal{L}_X$ is the Lie derivative in the direction of $X$. 
	\subsection{Ricci-Soliton Equation and Derivations}{\label{Ricci-Soliton Equation and Derivations}}
	Given a metric Lie algebra $(\mathfrak{g},g)$ and an orthonormal frame $\{X_1, \ldots, X_n\}$, the Ricci tensor $\ric_g$ has a matrix representation with respect to the frame, $\begin{bmatrix} \ric_g(X_i,X_j) \end{bmatrix}_{n \times n}$, as shown in examples~\ref{Ric3matrix} and~\ref{Ric4matrix}. Let $(\mathfrak{g}, g)$ be a nilpotent metric Lie algebra. Whenever a metric $g$ admits a Ricci tensor of the form $\ric_g \in \bb R I + \text{Der}(\mathfrak{g}) \in M_{n\times n}(\bb R)$, where $\text{Der}(\mathfrak{g})$ is the set of derivations on $\mathfrak{g}$, we say that $g$ is a Ricci nilsoliton. By Corollary 2 of~\cite{MR3268781} and Theorem 1.1 of~\cite{MR3263520}, $g$ is a Ricci nilsoliton if and only if $g$ is a Ricci soliton.
	\noindent For example the Ricci tensor~\ref{Ric3matrix} can be represented as 
	\begin{equation}
		\frac{\lambda_3^2}{2}\begin{bmatrix}-1 &  0& 0 \\ 0 & -1 & 0 \\ 0 & 0 & 1 \end{bmatrix} = -\frac{3\lambda_3^2}{2}I + \frac{\lambda_3^2}{2}\begin{bmatrix}2 & 0 & 0 \\ 0 & 2 &  0 \\ 0 & 0 & 4 \end{bmatrix} = \frac{\lambda_3^2}{2}\left[-3I + D\right]. \label{RicSolh3matrix}
	\end{equation}
	
	\noindent where $D$ is a derivation. The vector fields which satisfy~(\ref{ricsol}) for $(\mathfrak{h}^3, g)$ and $(\mathfrak{h}^4,g')$ cannot be left invariant~\cite[pg. 231]{MR3263424}. By Theorem~\ref{splitting} if $(\mathfrak{g},g)$ is a metric Lie algebra with an orthonormal Milnor frame, then $(\mathfrak{g},g)$ is a Ricci nilsoliton where the vector field which satisfies equation~(\ref{ricsol}) is not left-invariant.
	
	\begin{example}{\label{RSh3}}
		\normalfont Let $(\mathfrak{h}^3,g)$ be a metric Lie algebra with an orthonormal Milnor frame whose nontrivial structure constant is $\lambda_3 \neq 0$. By Theorem 4.3 of~\cite{MR425012} the matrix representation for the Ricci quadratic form is
		$$
		\ric_g = 
		\frac{\lambda_3^2}{2}\begin{bmatrix} 
			-1 & 0 & 0 \\
			0 & -1 & 0 \\
			0 & 0 & 1 
		\end{bmatrix}.
		$$
		The above matrix may be written as 
		$$
		\frac{\lambda_3^2}{2}\begin{bmatrix} -1 & 0 & 0 \\	0 & -1 & 0 \\	0 & 0 & 1 \end{bmatrix} = -\frac{3\lambda_3^2}{2}I + \lambda_3^2\begin{bmatrix}1 & 0 & 0 \\ 0 & 1 & 0 \\ 0 & 0 & 2 \end{bmatrix} = -\frac{3\lambda_3^2}{2}I + D
		$$
		where $D$ is a derivation. 
	\end{example}
	Suppose $(\mathfrak{h}^4, g)$ has an orthonormal Milnor frame. We shall show precisely when $g$ is a Ricci nilsoliton. 
	\begin{theorem}{\label{abs}}
		Suppose $(\mathfrak{h}^4,g)$ admits an orthonormal Milnor frame $\{X_1,\ldots, X_4\}$ with nontrivial structure constants $\lambda_3, \lambda_4 \neq 0$. Then $g$ is a Ricci nilsoliton if and only if $|\lambda_3| = |\lambda_4|$. 
	\end{theorem}
	
	\begin{proof}
		If $|\lambda_3| = |\lambda_4|$ then $\lambda_3^2 = \lambda_4^2$. The matrix representation of $\ric_g$ with respect to the frame $\{X_1, \ldots, X_4\}$ can be written as 
		\begin{alignat}{1}
			2\ric_g &= \begin{bmatrix}
				-\lambda_3^2 & 0 & 0 & 0 \\
				0 & -\lambda_3^2 - \lambda_4^2 & 0 & 0 \\
				0 & 0 & \lambda_3^2 -\lambda_4^2& 0 \\
				0 & 0  & 0 & \lambda_4^2 
			\end{bmatrix} = \begin{bmatrix} -\lambda_3^2 & 0 & 0 & 0 \\ 0 & -2\lambda_3^2 & 0 & 0 \\ 0 & 0 & 0 & 0 \\ 0 & 0 & 0 & \lambda_3^2\end{bmatrix}  \notag \\
			&= \lambda_3^2\begin{bmatrix}-1 & 0 & 0 & 0 \\ 0 & -2 & 0 & 0 \\ 0 & 0 & 0 & 0 \\ 0 & 0 & 0 & 1\end{bmatrix} \label{abs1}
		\end{alignat}
		where 
		\begin{equation}
			\begin{bmatrix}-1 & 0 & 0 & 0 \\ 0 & -2 & 0 & 0 \\ 0 & 0 & 0 & 0 \\ 0 & 0 & 0 & 1\end{bmatrix} = -3I + \begin{bmatrix}2 & 0 & 0 & 0 \\ 0 & 1 & 0 & 0 \\ 0 & 0 & 3 & 0 \\ 0 & 0 & 0 & 4 \end{bmatrix} = -3I + D \label{abs2}
		\end{equation}
		and so by~(\ref{abs1}) and~(\ref{abs2}) $2\ric_g = -3\lambda_3^2 I + \lambda_3^2D$ where $D$ is a derivation. Thus $g$ is a Ricci nilsoliton. 
		
		\medskip
		Now suppose that $\ric_g$ is a Ricci nilsoliton. Let $D$ be a derivation such that $\ric_g = cI + D$ for some constant $c \in \bb R$. Because $\ric_g$ is diagonalized with respect to the frame $\{X_1,\ldots, X_4\}$, $D$ must be diagonalized. For each $i$, let $D_i \in \bb R$ so that $DX_i = D_iX_i$. Then 
		\begin{alignat}{1}
			\lambda_3D_3X_3 &= D(\lambda_3X_3)  = D[X_1,X_2] = [DX_1, X_2] + [X_1,DX_2] \notag \\
			&= (D_1 + D_2)[X_1,X_2] = \lambda_3(D_1 + D_2)X_3
		\end{alignat}
		and so $D_3 = D_1 + D_2$. Similarly 
		$$
		\lambda_4D_4X_4 = D[X_2,X_3] = [DX_2,X_3] + [X_2,DX_3] = \lambda_4(D_2 + D_3)X_4
		$$
		and so $D_4 = D_2 + D_3 = D_1 + 2D_2$. We obtain the system of equations 
		\begin{alignat}{1}
			-\lambda_3^2 &= c + D_1 \notag \\
			-\lambda_3^2 - \lambda_4^2 &= c + D_2 \notag \\
			\lambda_3^2 - \lambda_4^2 &= c + D_1 + D_2 \notag \\
			\lambda_4^2 &= c + D_1 + 2D_2
		\end{alignat}
		which gives us the equation 
		$$
		\begin{bmatrix}-1 & -1 & 0 \\ -2 & -1 & -3 \\ 2 & 2 & 3\end{bmatrix}\begin{bmatrix}c \\ D_1 \\ D_2 \end{bmatrix} = \lambda_3^2\begin{bmatrix}1 \\ 1 \\ 1\end{bmatrix} 
		$$
		and so 
		$$
		\begin{bmatrix}c \\ D_1 \\ D_2 \end{bmatrix} = \frac{\lambda_3^2}{3}\begin{bmatrix}-3 & -3 & -3 \\ 0 & 3 & 3 \\ 2 & 0 & 1\end{bmatrix}\begin{bmatrix}1 \\ 1 \\1 \end{bmatrix} = \lambda_3^2\begin{bmatrix}-3 \\ 2 \\ 1\end{bmatrix}.
		$$
		Thus $\lambda_4^2 = c + D_1 + 2D_2 = \lambda_3^2( -3 + 2 + 2) = \lambda_3^2$ which implies $|\lambda_4|= |\lambda_3|$. 
	\end{proof}
	
	Now we may generalize to any metric Lie algebra with an orthonormal Milnor frame. 
	\begin{corollary}{\label{CorollaryRicciSolitonh4}}
		Let $(\mathfrak{g}, g)$ be a Lie algebra with an orthonormal Milnor frame $\{X_1, \ldots, X_n\}$ whose structure constants are $\lambda_1, \ldots, \lambda_n \in \bb R$.
		The metric $g$ is a Ricci nilsoliton if and only if for any $i$ such that $\text{Span}\{X_i, X_{i+1}, X_{i+2}, X_{i+3} \} \cong \mathfrak{h}^4$, $|\lambda_{i+2}| = |\lambda_{i+3}|$. 
	\end{corollary}
	
	In Section~\ref{Ricci-Soliton Equation and Derivations}, we determined that the set of metrics which admit an orthonormal Milnor frame $\{X_1,\ldots, X_4\}$ whose nontrivial structure constants are $\lambda_3,\lambda_4 \neq 0$ such that $|\lambda_3| = |\lambda_4|$ are Ricci nilsolitons. A result of Theorems~\ref{orthonormalh4} and \ref{abs} shall classify all metrics on $\mathfrak{h}^4$ which are Ricci nilsolitons. 
	
	\begin{corollary}
		Let $g$ be a metric on $\mathfrak{h}^4$. Then $g$ is a Ricci nilsoliton if and only if there exists an orthonormal Milnor frame $\{X_1, \ldots, X_4\}$ with nontrivial structure constants $\lambda_3,\lambda_4 \neq 0$ such that $|\lambda_3| = |\lambda_4|$, i.e if and only if $(\mathfrak{h}^4,g)$ has an orthonormal Milnor frame where the Ricci signature is $(--,0,+)$. 
	\end{corollary}
	
	\begin{proof}
		Let $\{F_1, \ldots, F_4\}$ be an orthonormal frame for the metric Lie algebra $(\mathfrak{h}^4, g)$ whose structure constants are $C_{2,4}^1 = a > 0$, $C_{3,4}^1 = b \in \bb R$ and $C_{3,4}^2 = c > 0$. The matrix representation of $\ric_g$ with respect to the frame $\{F_1,\ldots, F_4\}$ is of the form 
		\begin{equation}{\label{Riccimatrixnonorthogonal}}
			2\ric_g = 
			\begin{bmatrix}
				a^2 + b^2  & bc & 0 & 0 \\ bc & c^2 - a^2 & -ab & 0 \\ 0 & -ab & -b^2 - c^2 & 0 \\ 0 & 0 & 0 & -a^2 - b^2 - c^2
			\end{bmatrix}.
		\end{equation}
		Suppose that $\ric_g$ is a Ricci nilsoliton. Then there exists $k \in \bb R$ and $D \in \text{Der}(\mathfrak{g})$ such that $D = \ric_g - kI$. Thus $DF_1 = (a^2+b^2-k)F_1 + bcF_2$. Because $\ad_{F_1} = 0$, 
		$$
		0 = D[F_1,F_4] = [DF_1, F_4] + [F_1,DF_4] = (a^2 + b^2-k)[F_1,F_4] + bc[F_2,F_4] = abcF_1
		$$
		where $a,c> 0$ so that $b = 0$. By Theorem~\ref{orthonormalh4} $(\mathfrak{h}^4,g)$ must have an orthonormal Milnor frame $\{X_1, \ldots, X_4\}$ with nontrivial structure constants $\lambda_3,\lambda_4 \neq 0$. By Theorem~\ref{abs}, $|\lambda_3| = |\lambda_4|$ and so the Ricci signature is of the form $(--,0,+)$.  
	\end{proof}
	It is not the case that the set of metrics which admit a Ricci signature of the form $(-,-,0,+)$ admit an orthonormal Milnor frame. By Lemma 3 of~\cite{MR2655935}, we can find a metric which admits an orthonormal frame $\{f_1, \ldots, f_4\}$ such that  $\text{span}\{f_1, \ldots, f_4\} = \mathfrak{h}^4$ and has nontrivial structure constants $C_{2,4}^1 = C_{3,4}^2 > 0$ and $C_{2,4}^1\neq 0$. Because $C_{2,4}^1 \neq 0$, $(\mathfrak{h}^4,g)$ does not admit an orthonormal Milnor frame.
	
	\begin{example}
		\normalfont Let $\mathfrak{h}^4$ be a Lie algebra with a frame $\{f_1, \ldots, f_4\}$ with nontrivial structure constants $C_{2,4}^1, C_{3,4}^1,C_{3,4}^2 = 1$. Existence of such a Lie algebra is given by Lemma 3 of~\cite{MR2655935}. Let $g$ be a metric which makes $\{f_1, \ldots, f_4\}$ an orthonormal frame. The Ricci tensor is of the form 
		\begin{equation}{\label{notriccisolitonexample}}
			2\ric_g = \begin{bmatrix}2 & 1 & 0 & 0 \\ 1 & 0 & -1 & 0 \\ 0 & -1 & -2 & 0 \\ 0 & 0 & 0 & -3\end{bmatrix}.
		\end{equation}
		The characteristic polynomial of the above matrix is $x(x+3)(x^2-6)$ and so the signature of $\ric_g$ in~(\ref{notriccisolitonexample}) is $(-,-,0,+)$. If $g$ is a Ricci nilsoliton then there exists $k \in \bb R$ such that $D:=\ric_g - kI$ is a derivation. Because $[f_1,f_4] = 0$, $D[f_1,f_4] = 0$ and so 
		$$
		0 = D[f_1,f_4] = [Df_1,f_4] + [F_1,Df_4] = [(2-k)f_1 + f_2, f_4] + [f_1, -3f_4] = f_1
		$$
		which is a contradiction. Thus $g$ must not be a Ricci nilsoliton. 
	\end{example} 
	\newpage
	\bibliographystyle{plain}
	\bibliography{references.bib}
	
\end{document}